\documentclass[11pt]{amsart}

\usepackage{amssymb}
\usepackage{url}
\usepackage{amscd}
\usepackage{texdraw}
\usepackage{color}

\theoremstyle{plain}
\newtheorem{theorem}{Theorem}

\newtheorem{property}{Property}

\newtheorem{lemma}{Lemma}
\newtheorem{remark}{Remark}
\newtheorem{proposition}{Proposition}
\newtheorem{definition}{Definition}

\input txdtools

\begin{document}
\title[Canonical Thurston Obstructions]
{Canonical Thurston Obstructions for Sub-Hyperbolic Semi-Rational Branched Coverings}
\author{Tao Chen and Yunping Jiang}

\address{Department of Mathematics,
CUNY Graduate Center, 365 Fifth Avenue, New York, NY 10016}
\email{chentaofdh@gmail.com}

\address{Department of Mathematics,
CUNY Queens College and Graduate Center, Flushing, NY 11367}
\email{yunping.jiang@qc.cuny.edu}

\thanks{}


\begin{abstract} We prove that the canonical Thurston obstruction
for a sub-hyperbolic semi-rational branched covering exists if the
branched covering is not CLH-equivalent to a rational
map.
\end{abstract}

\maketitle

\section{Introduction}

Let $S^{2}$ be the two-sphere. We use $\widehat{\mathbb C}$
to denote the Riemann sphere which is $S^{2}$ equipped with
the standard complex structure. 
All maps in this paper are orientation-preserving.

Let $f: S^{2} \to S^{2}$ be a branched covering of degree $d\ge 2$. Let
$$
C_{f} = \{x \in S^{2}\:|\: \deg_{x} f \ge 2\}
$$
denote the set of the critical points of $f$ and
$$
P_{f}=\overline{ \bigcup_{k\ge 1} f^{k}(C_{f})}
$$
denote the post-critical set of $f$.

We say $f$ is {\em critically finite} if
$\sharp P_{f}$ is finite. We say $f$ is {\em geometrically finite} if
$\sharp P_{f}$ is infinite but the accumulation set $P_{f}'$ of $P_{f}$
is a finite set.

\vspace*{5pt}
\begin{definition}~\label{Comb}
Suppose $f, g: S^{2} \to S^{2}$ are two branched coverings of degree $d\geq 2$.
They are said to be combinatorially equivalent if there exists a pair of
homeomorphisms $\phi,\ \varphi:  S^{2}\to S^{2}$ such that
\begin{itemize}
\item[a)] $\phi$ is isotopic to $\varphi$ rel $P_f$ and
\item[b)] $\phi\circ f=g\circ \varphi$.
\end{itemize}
\end{definition}

Note that in a) of Definition~\ref{Comb}, the statement that $\phi$
is isotopic to $\varphi$ rel $P_f$ means that there is a continuous map
$H(x,t): S^{2}\times [0,1] \to S^{2}$ such that
\begin{enumerate}
\item for each $t\in [0,1]$, $H_{t}(x)=H(x,t): S^{2} \to S^{2}$ is a homeomorphism;
\item $H_{0}=\phi$ and $H_{1}=\varphi$;
\item for any point $y\in P_f$ and any $t\in [0,1]$, $H_t(y)=\phi(y)=\varphi(y)$.
\end{enumerate}

Suppose $f: S^2\rightarrow S^2$ is critically finite. Then
there is an orbifold structure on $S^{2}$ associated to $f$ as follows.
Define the signature $\nu_f: S^2 \rightarrow \mathbb{Z}_+\cup \{\infty\}$ as that
$\nu_f(x)$ is the least common multiple of local degrees $\deg_{y} f^{n}$ over $y\in f^{-n}(x)$ for all $n\geq 1$.
The \emph{orbifold} associated to $f$ is $\Omega_{f} =(S^{2}, \nu_{f})$. The Euler characteristic of $\Omega_f$, by definition, is
$$
\chi(\Omega_f)=2-\Sigma_{x\in S^2}\Big(1-\frac{1}{\nu_f(x)}\Big).
$$
It is known that $\chi (\Omega_{f}) \leq 0$ (see Proposition 9.1 (i) in \cite{DH}).
Moreover, the orbifold $\Omega_f$ is called \emph{hyperbolic} if $\chi(\Omega_f)<0$ and \emph{parabolic} if $\chi (\Omega_{f}) = 0$.

\vspace*{5pt}

\renewcommand{\thetheorem}{\Alph{theorem}}
\begin{theorem}[\cite{Th,DH}]~\label{Thurston}
Suppose $f$ is a critically finite branched covering with
a hyperbolic orbifold $\Omega_{f}$. Then $f$ is combinatorially equivalent to a
rational map $R$ if and only if $f$ has no Thurston obstructions.
Moreover, the rational map $R$ is unique up to conjugations by
automorphisms of the Riemann sphere.
\end{theorem}

The reader can refer to \S2 for the definition of Thurston obstruction.
If a critically finite branched covering $f$ is not equivalent
to a rational map, then there must exist Thurston obstructions.
The canonical Thurston obstruction is the most interesting
one among all Thurston obstructions. The reader can 
refer to \S2 for the term non-peripheral and \S4 for the definition of $l(\gamma,x)$.

\vspace*{5pt}
\begin{theorem}[\cite{Pi}]~\label{Pilgrim}
Suppose $f$ is a critically finite branched covering
with a hyperbolic orbifold $\Omega_{f}$, and
let $\Gamma_c$ denote the set of all homotopy class of
non-peripheral curves $\gamma$ in $S^2\setminus P_f$
such that $l(\gamma, x_n)\to 0$ as $n\to \infty$. Then
\begin{itemize}
\item[(1)] $\Gamma_c$ is empty, and $f$ is combinatorially equivalent to a rational map;
\item[(2)] Otherwise, $\Gamma_c$ is a Thurston obstruction and hence is a canonically
defined Thurston obstruction to the existence of a rational map.
\end{itemize}
\end{theorem}

The nonexistence of Thurston's obstruction condition is essentially true
for any rational map.

\vspace*{5pt}
\begin{theorem}[\cite{Mc}]~\label{McMullen}
Suppose $R: \hat{\mathbb C}\to \hat{\mathbb C}$ is a rational map.
Let $\Gamma$ be a multi-curve on $\hat{\mathbb C}-P_{R}$. It can be
a Thurston obstruction only in the following cases:
\begin{itemize}
\item[1)] $R$ is critically finite with $\#P_{R}=4$
and the orbifold $\Omega_{R}=(S^{2}, (2,2,2,2))$ is parabolic.
And, moreover, $R$ is a double covered by an integral torus endomorphism (it is a special case of a Latt\'es map).
\item[2)] $P_{R}$ is an infinite set and $\Gamma$ includes the essential
curves in a finite system of annuli permuted by $R$. These annuli
lie in Siegel disks or Herman rings for $R$ and each annulus is a
connected component of $\widehat{\mathbb C}-P_{R}$.
\end{itemize}
\end{theorem}

The reader can refer to~\cite{Mi} for a definition of a Latt\'es map and
for definitions of a Siegel disk and a Herman ring.

For a geometrically finite branched covering $f$,
the situation is much more complicated.
It was first studied in a manuscript~\cite{CJS}, which was divided into
two parts~\cite{CJS1} and~\cite{CJS2}. The first part was eventually completed and published 
in~\cite{CJ} as follows.

\vspace*{5pt}
\begin{theorem}[\cite{CJ}]~\label{CuiJiang}
There is a geometrically finite branched covering such that it
has no Thurston obstruction and it is not combinatorially equivalent to any rational map.
\end{theorem}

Due to this theorem, a semi-rational branched covering and a sub-hyperbolic semi-rational branched covering
are introduced in~\cite{CJ} among the space of all geometrically finite branched coverings:

\vspace*{5pt}
\begin{definition}~\label{semirational}
Suppose $f: \hat{\mathbb C}\to \hat{\mathbb C}$
is a geometrically finite branched covering of degree $d\geq 2$.
We say $f$ is {\it semi-rational} if
\begin{itemize}
\item[(1)] $f$ is holomorphic in a neighborhood of $P_{f}'$;
\item[(2)] each cycle $<p_{0}, \cdots, p_{k-1}>$ of period $k\geq 1$ in $P_{f}'$ is either attractive,
that is, $0<|(f^{k})'(p_{0})|<1$, or super-attractive, that is, $(f^{k})'(p_{0})=0$,
or parabolic, that is, $|(f^{k})'(p_{0})|=1$ and $\big((f^{k})'(p_{0})\big)^{q}=1$ for some integer $q\geq 1$; and
\item[(3)] each attracting petal associated with a
parabolic cycle in $P_{f}'$ contains a point in the post-critical set $P_{f}$.
\end{itemize}
Furthermore, if all cycles in $P_{f}'$ are either attractive or
super-attractive, we call $f$ a {\it sub-hyperbolic semi-rational
} branched covering.
\end{definition}

Clearly, every geometrically finite rational map is a semi-rational branched covering. Furthermore, we have that

\vspace*{5pt}
\begin{theorem}[\cite{CJ}]~\label{equivcont}
A semi-rational branched covering $f$ is always combinatorially
equivalent to a sub-hyperbolic semi-rational branched covering $g$.
\end{theorem}

Thus, to study the combinatorial classification in the space of all semi-rational geometrically finite branched coverings,
it is enough to study all sub-hyperbolic semi-rational branched coverings. Therefore, the CLH (combinatorially and
locally holomorphically) equivalence was introduced in~\cite{CJ} in the space of
all sub-hyperbolic semi-rational branched coverings.

\vspace*{5pt}
\begin{definition}~\label{CLH}
Suppose $f$ and $g$ are two sub-hyperbolic semi-rational branched
coverings. We say that they are CLH-equivalent if there exists a pair of
homeomorphisms $\phi,\ \varphi:  \widehat{\mathbb C}\to
\widehat{\mathbb C}$ such that
\begin{enumerate}
\item $\phi$ is isotopic to $\varphi$ rel $P_f$,
\item $\phi\circ f=g\circ \varphi$, and
\item $\phi|U_f=\varphi|U_f$ is holomorphic on some open set
$U_f\supset P_{f}'$.
\end{enumerate}
\end{definition}

We then completed the second part of study.

\vspace*{5pt}
\begin{theorem}[\cite{CJS2}\footnote{This paper was rewritten by Cui and Tan recently~\cite{CT}.}, \cite{JZ}]~\label{ZhangJiang}
Suppose $f$ is a sub-hyperbolic semi-rational branched covering.
Then $f$ is CLH-equivalent to a rational map $R$ if and only if $f$
has no Thurston obstructions. In this case, the rational map $R$ is
unique up to conjugations by automorphisms of the Riemann sphere.
\end{theorem}

Thus, the study of canonical Thuston obstructions for sub-hyperbolic semi-rational branched coverings becomes
our final goal to have a complete understanding of combinatorial structures for geometrically finite branched coverings.
In this paper we will complete our final goal.

For the critically finite case, the Teichm\"uller space of the Riemann sphere
minus several points was considered in \cite{DH}. A crucial technical theorem they proved is that a sequence $\{x_n\}$ in
the Teichm\"uller space converges if and only if its projection is a precompact set in the moduli space.
The Mumford compactness theorem applies. But neither of them applies for the geometrically finite case.
Therefore, we turn to study the bounded geometry property, which makes our approach different 
from the method applied in \cite{DH}.
Roughly speaking, our main result in this paper is that
if a sub-hyperbolic semi-rational branched covering $f$ is not CLH-equivalent
to any rational map, then there exists a canonical Thurston obstruction.
To have a more precise statement of our main result, let us first give
an idea of a proof of Theorem~\ref{ZhangJiang} by using bounded geometry as follows.

Suppose $f$ is a sub-hyperbolic semi-rational branched covering.
Let $P'_f=\{a_i\}$ be the set of accumulation points of $P_{f}$. Then
every $a_i$ is periodic. There exists a collection of finite number
of open disks
\begin{equation}~\label{pd}
\Lambda=\{D_i\}
\end{equation}
centered at $\{a_i\}$ and a collection of finite number
of annuli $\{A_i\}$ (we call them the shielding rings) such that

\begin{itemize}
\item[(i)] $\overline{A_i}\cap P_f=\emptyset$;
\item[(ii)] $A_{i}\cap D_{i}=\emptyset$, but one components of $\partial A_i$ is the boundary of $D_i$;
\item[(iii)]
$(\overline{D_i\cup A_{i}})\cap (\overline{D_j\cup
A_{j}})=\emptyset\ \ \text{for}\ \ i\neq j$;
\item[(iv)] $f$ is holomorphic on $\overline{D_i}\cup A_i$; and
\item[(v)] every $f(\overline{D_i}\cup A_i)$ is contained in
$D_{i+1}$ for $1 \le i \le k-1$ and  $f(\overline{D_k}\cup A_k)$ is contained in $D_{1}$
where $k$ is the period of $a_{i}$.
\end{itemize}

Denote $D=\cup_i D_i$ and
\begin{equation}~\label{p1}
P_1=P_f\setminus D.
\end{equation}
Since $a_{i}$ are accumulation point of $P_f$, it follows that $\sharp P_1$ is finite. 
Without loss of generality, we assume that $0$, $1$, and $\infty$ belong to $P_1$.
Define
\begin{equation}~\label{defq}
Q=P_1\cup \overline{D}\ \  \text{and}\ \ X=\partial Q=P_1\cup
\partial D.
\end{equation}
We associate with $f$ the Teichmuller space $T_{f} =T(\widehat{\mathbb{C}}\setminus Q, X)$
which is the Teichm\"uller space of Riemann surface
$\widehat{\mathbb{C}}\setminus Q$ whose boundary is $X$. Note that
$T_f$ is also the Teichm\"uller space $T_{0}(Q)$ which is the space
of all $Q$-equivalent classes of all Beltrami coefficients $\mu$ on
$\widehat{\Bbb C}$ such that $\mu|Q=0$. (Two Beltrami coefficients
$\mu$ and $\nu$ are $Q$-equivalent if the normalized quasiconformal
maps $w^{\mu}$ and $w^{\nu}$ are isotopic rel $Q$.) The space $T_{f}$ is
a complex manifold. The Teichm\"uller metric and the Kobayashi
metric on $T_{f}$ are also equal (refer to, for example,~\cite{EM,GJW,JMW}).

The map $f$ induces a holomorphic map $\sigma_{f}$ from $T_{f}$ into
itself and $\sigma_{f}$ weakly contracts the Teichm\"uller metric. An equivalent statement
of Theorem~\ref{ZhangJiang} is that $\sigma_f$ has a unique fixed point
if and only if $f$ has no Thurston obstruction.

Every point $x$ in $T_{f}$ determines a complex structure
on $\widehat{\mathbb C}\setminus Q$ up to homotopy. Then $(\widehat{\mathbb C}\setminus Q, x)$
is a Riemann surface $R_x$.
We embed $R_{x}$ into the Riemann sphere $\widehat{\mathbb C}$ by a quasiconformal
map $\phi_{x}:\widehat{\mathbb C} \to \widehat{\mathbb C}$ fixing $0$, $1$, $\infty$.
Then $\widehat{\mathbb C}\setminus \phi_{x}(Q)$ is a representative of $R_{x}$. The reader can refer to \S4.

Let $d(\cdot, \cdot)$ mean the spherical distance on $\widehat{\mathbb C}$. We define a subspace
${\mathcal T}_{f,b}$ of $T_f$ for each $b>0$ as follows:

\vspace*{5pt}
\begin{definition}~\label{bgg}
Let $b>0$ be a constant. Let ${\mathcal T}_{f,b}$ be the
subspace of $x=[\mu]\in {\mathcal T}_{f}$ satisfying the following
conditions:
\begin{enumerate}
\item for all $z_{i}\not=z_{i'} \in P_{1}$,
$d(\phi_{\mu}(z_{i}), \phi_{\mu}(z_{i'}))\geq b$;
\item for all $z_{j}\in P_{1}$ and all $D_{i}\in \Lambda$,
$d(\phi_{\mu}(z_{j}), \phi_{\mu}(D_{i}))\geq b$;
\item for all $D_{i}\not= D_{i'}\in \Lambda$,
$d(\phi_{\mu}(D_{i}), \phi_{\mu}(D_{i'}))\geq b$;
\item every $D_{i}\in \Lambda$, $\phi_{\mu}(D_{i})$
contains a round disk of radius $b$ centered at $\phi_{\mu}(c_{i})$.
\end{enumerate}
We call ${\mathcal T}_{f,b}$ the subspace having the bounded
geometry property determined by $b$.
\end{definition}

Take an arbitrary $x_{0} \in T_f$ and let $x_n=\sigma_{f}^{n}(x)$.
If $f$ has no Thurston obstructions,
we can prove that $\{x_{n}\}_{n=0}^{\infty}\subset {\mathcal T}_{f,b}$ for some $b>0$.
This implies that the sequence $\{x_{n}\}_{n=0}^{\infty}$ converges in $ T_{f}$.
Thus $\sigma_{f}$ has a unique fixed point, and $f$ is CLH-equivalent to
a unique sub-hyperbolic rational map.

For a non-peripheral curve $\gamma$ in $\widehat{\mathbb C}\setminus Q$, let
$l(\gamma,x)$ denote the hyperbolic length of the unique geodesic in $R_{x}$
which is homotopic to $\gamma$ in $\widehat{\mathbb C} \setminus Q$.
If $\{x_n\}\subset {\mathcal T}_{f,b}$ for some $b>0$, then there is a $\delta>0$ such that
$l(\gamma,x_{n}) \ge \delta$ for any non-peripheral curve $\gamma$ in $\widehat{\mathbb C}\setminus Q$
and any $n \ge 0$. Therefore, if $f$ is not CLH-equivalent to a sub-hyperbolic rational map,
then there is a sequence of non-peripheral curves $\gamma_{n}$ in $\widehat{\mathbb C} \setminus Q$
such that $l(\gamma_{n},x_n) \to 0$ as $n \to \infty$.

\vspace*{5pt}
\textbf{Question:}
Suppose $f$ is not CLH-equivalent to a rational map.
Does there exist a non-peripheral curve $\gamma$,
such that for any $x_0\in T_f$ and
$x_n=\sigma_f^n(x_0)$, $n>0$,
$l(\gamma,x_n)\to 0$ as $n\to \infty$?

\vspace{3pt}
We gives an affirmative answer to this question.
The positive answer to this question gives a way how $x_{n}$ tends to the boundary of $T_{f}$.
More precisely, we will prove a stronger result as follows.

\vspace*{5pt}
\addtocounter{theorem}{-6}
\renewcommand{\thetheorem}{\arabic{theorem}}
\begin{theorem}[Main Theorem]~\label{ChenJiang}
Suppose $f$ is a sub-hyperbolic semi-rational branched covering.
Let $\Gamma_c$ denote the set of all homotopy
classes of non-peripheral curves $\gamma$ in $\widehat{\mathbb C} \setminus Q$
such that $l(\gamma, x_{n})\rightarrow 0$ as $n\rightarrow \infty$ for any
initial $x_0\in T_{f}=T_{0} (Q)$.
Then we have that either
\begin{itemize}
\item[(a)] $\Gamma_c=\emptyset$, then $f$ is CLH-equivalent
to a sub-hyperbolic rational map, or
\item [(b)] $\Gamma_c\not= \emptyset$ is a Thurston
obstruction for $f$ and $f$ is not CLH-equivalent to a rational map. In this
case, we call $\Gamma_{c}$ the canonical Thurston obstruction for $f$.
\end{itemize}
\end{theorem}

\vspace*{5pt}
The paper is organized as follows.
In \S2, we define Thurston obstructions for sub-hyperbolic semi-rational branched coverings.
In \S3, we review non-negative matrices and study some properties for an irreducible non-negative matrices.
In \S4, we study the Teichm\"uller space associated with a sub-hyperbolic semi-rational branched covering and short geodesics.
For any Thurston obstruction $\Gamma$, we can decompose it into $\Gamma_{0}$ and $\Gamma_{\infty}$ (see Definition~\ref{depth}).
We estimate the upper bound for $\Gamma_{\infty}$ in \S5 and the lower bound for $\Gamma_{0}$ in \S6.
Finally, we prove Theorem~\ref{ChenJiang} in \S7.

\vspace*{20pt}
\noindent {Acknowledgement:} We would like to thank Professors Zhe Wang,
Xiaojun Huang, and Yunchun Hu for many endless hours to listen to our talks
and to help to check arguments and to provide extremely useful comments.
We would like to thank Professors Sudeb Mitra and Linda Keen to read our initial version of this paper and to provide
usefull suggestions and comments.

\vspace*{5pt}

\section{Thurston obstructions}

Suppose $f$ is a sub-hyperbolic semi-rational branched covering.
Let $Q$ be the set as we defined in (\ref{defq}). Then
$$
f: \widehat{\mathbb{C}}\setminus f^{-1}(Q)\longrightarrow  \widehat{\mathbb{C}}\setminus Q
$$
is a covering map of finite degree.
If $\gamma$ is a simple closed curve in $ \widehat{\mathbb{C}}\setminus Q$,
then all the components of $f^{-1}(\gamma)$ are simple closed
curves in $ \widehat{\mathbb{C}}\setminus f^{-1}(Q)$, which is a subset of $ \widehat{\mathbb{C}}\setminus Q$.
Thus all the components of $f^{-1}(\gamma)$ are simple closed
curves in $ \widehat{\mathbb{C}}\setminus Q$.

A simple closed curve $\gamma$ is said to be \emph{non-peripheral}
if each component of $ \widehat{\mathbb{C}}\setminus \gamma$ contains at least two points
of $Q$. A \emph{multi-curve}
\begin{equation}~\label{muti}
\Gamma=\{\gamma_1,\ \cdots,\ \gamma_n\}
\end{equation}
is a set of finitely many pairwise disjoint, non-homotopic,
and non-peripheral curves in $ \widehat{\mathbb{C}}\setminus Q$.
For each multi-curve $\Gamma$ in (\ref{muti}), let
$$
\mathbb{R}^{\Gamma}=<\gamma_1, \cdots, \gamma_{n}>
$$
be the real vector space of dimension $n$ with a basis $\Gamma$.
We define a linear transformation
$$
f_\Gamma: \mathbb{R}^\Gamma\rightarrow \mathbb{R}^\Gamma
$$
as follows: For each
$\gamma_{j}\in \Gamma$, let $\gamma_{i,j,\alpha}$ denote the
components of $f^{-1}(\gamma_j)$ homotopic to $\gamma_i$ in
$ \widehat{\mathbb{C}}\setminus Q$ and $d_{i,j,\alpha}$ be the degree
of $f|_{\gamma_{i,j,\alpha}}:\ \gamma_{i,j,\alpha}\rightarrow
\gamma_j$. Define
$$
f_{\Gamma}(\gamma_j)=\Sigma_i \Big( \Sigma_\alpha \frac{1}{d_{i,j,\alpha}} \Big)\gamma_i.
$$
Let $A_\Gamma$ be the corresponding matrix, that is
$$
f_{\Gamma} {\bf v} =A_{\Gamma} {\bf v}, \quad {\bf v}\in \mathbb{R}^{\Gamma}.
$$
Since the matrix $A_{\Gamma}$ is non-negative, by the Perron-Frobenius Theorem,
there exists a maximal non-negative eigenvalue $\lambda(A_\Gamma)$ which
is the spectral radius of $A_\Gamma$.

A multi-curve $\Gamma$ is said to be \emph{$f$-stable} if for any $\gamma\in\Gamma$, every
non-peripheral component of $f^{-1}(\gamma)$ is homotopic to an
element of $\Gamma$ rel $Q$.

\begin{definition}
A stable multi-curve $\Gamma$ is called a \emph{Thurston obstruction} for $f$
if $\lambda(A_\Gamma)\geq 1$.
\end{definition}

\begin{remark}
The definition of a Thurston obstruction for the critically finite case is similar by replacing $Q$ by $P_{f}$.
\end{remark}

\section{Non-negative matrices}
Since a Thurston obstruction is determined by a non-negative matrix, we give a brief review
of some results in the matrix theory about non-negative matrices.
We use~\cite{Ga} as a reference.

A non-negative $n\times n$ matrix $A$ is called \emph{irreducible},
if no permutation of the indices places the matrix in a
block lower-triangular form. More precisely, there is no permutation matrix
$P$, which is a matrix consisting of $0$ and $1$ such that each row or
each column contains one and only one $1$, such that
$$
PAP^{-1}=\begin{pmatrix} A_{11} & 0\\
                         A_{21} & A_{22}
         \end{pmatrix},
$$
where $A_{11}$ and $A_{22}$ are square matrices. An equivalent definition of irreducibility is that for
any $1\leq i,\ j\leq n$, there exists a $0\leq q=q(i,j)\leq n$ such that
the $ij$-th entry of $A^{q}$ is positive.

For the $n$-dimensional vector space ${\mathcal V}$, we will use the norm
\begin{equation}~\label{norm}
\| {\bf v}\| =\max_{1\leq i\leq n} |v_{i}|, \quad {\bf v} =(v_{1}, \cdots, v_{n})\in {\mathcal V}
\end{equation}
in the rest of this paper.
For any linear map $L: {\mathcal V}\to {\mathcal V}$, let $A$ be the corresponding matrix for $L$,
define
$$
\| A\| =\sup_{\|{\bf v}\|=1} \|A{\bf v}\|.
$$
The spectral radius $\lambda (A)$ of $A$ can be calculated as
$$
\lambda (A) = \lim_{n\to \infty} \sqrt[n]{||A^{n}||}\geq 0.
$$

If $A$ is a non-negative, the Perron-Frobenius Theorem implies that
$\lambda(A)$ is an eigenvalue of $A$. Thus it is a maximal eigenvalue of $A$.
If $A$ is irreducible, $\lambda (A)$ is a simple, positive, maximal eigenvalue with a positive eigenvector
${\bf v} =(v_{1},\cdots, v_{n})$, i.e, $v_{i}>0$ for all $1\leq i\leq n$.
However, there may exist another eigenvalue $\mu\not=\lambda(A)$ but $|\mu|=\lambda (A)$.
For example, consider
$$
A=\begin{pmatrix} 0 & 1\\
                  1 & 0
  \end{pmatrix}.
$$
It is an irreducible matrix. The spectral radius is $1$ which is a simple, positive, maximal eigenvalue
with an eigenvector ${\bf v}_{1}=(1,1)$. However, $-1$ is also an eigenvalue with an eigenvector ${\bf v} =(1,-1)$.
But if $A$ is positive, that is, every entry is a positive number, the Perron-Frobenius theorem states that $\lambda (A)$ is a unique, simple, positive,
maximal eigenvalue with a positive eigenvector ${\bf v} =(v_{1},\cdots, v_{n})$, i.e, $v_{i}>0$ for all $1\leq i\leq n$. Here the term ``unique" means
that all other eigenvalues $\mu$ of $A$ satisfy that
$$
|\mu|<\lambda (A).
$$

\vspace*{5pt}
\begin{definition}~\label{red}
We say that a multi-curve $\Gamma$ is \emph{irreducible} if the corresponding
matrix $A_{\Gamma}$ of the linear map $f_{\Gamma}: {\mathbb R}^{\Gamma}\to {\mathbb R}^{\Gamma}$ is irreducible.
\end{definition}

For any non-negative matrix $A$, we can rearrange the order of the basis such that
\begin{equation}~\label{irredform}
A=\begin{pmatrix} A_{11} & 0 & \cdots & 0\\
A_{21}& A_{22}& \cdots & 0\\
\cdots & \cdots & \cdots& \cdots\\
A_{s1}& A_{s2}& \cdots & A_{ss}
\end{pmatrix}
\end{equation}
and all the blocks $A_{jj}$ on the diagonal are either irreducible or $0$ matrices.
It is not hard to calculate that
$$
\lambda(A)=\text{max}_{j}\lambda(A_{jj}).
$$

Now we consider $A=A_\Gamma$ as the corresponding matrix
of the linear map $f_{\Gamma}: {\mathbb R}^{\Gamma}\to
{\mathbb R}^{\Gamma}$ for an $f$-stable multi-curve $\Gamma=\{ \gamma_{1}, \cdots, \gamma_{n}\}$.
We assume that $A$ is in the form of (\ref{irredform}). Then we can use $\Gamma_j$ to denote
the subset of curves in $\Gamma$ corresponding to the $j$-th block in $A$. That is,
$A_{jj}=A_{\Gamma_j}$. It is clear that
$$
\Gamma=\cup_{j} \Gamma_{j}.
$$
We call $\{\Gamma_j\}$ irreducible decomposition of $\Gamma$. The reader should note that
$\Gamma_{j}$ may not be $f$-stable.

Denote
$$
\Gamma_{Ob}=\cup_{j}\Gamma_j
$$
where the union runs over all $j$ such that $\lambda(A_{jj})\geq 1$.
We have the following definition to relate every element in $\Gamma$ to $\Gamma_{Ob}$ if it is not empty.

\vspace*{5pt}
\begin{definition}~\label{depth}
Suppose $\Gamma$ is an $f$-stable multi-curve. For every $\gamma \in \Gamma$, if
there exists a $\gamma_{ob}\in \Gamma_{Ob}$ and an integer $k\geq 0$ such
that $\gamma$ is homotopic to a component $f^{-k}(\gamma_{ob})$, then we define
the depth of $\gamma$ with respect to $\Gamma$ to be the least
such integer $k$. Otherwise, we define the depth as $\infty$.
The set of all elements in $\Gamma$ with finite depth is denoted by $\Gamma_0$.
The set of all elements with infinite depth is denoted by $\Gamma_\infty$.
\end{definition}

Then
$$
\Gamma =\Gamma_{0}\cup \Gamma_{\infty}.
$$
It is clear that if $\Gamma$ is a Thurston obstruction, then $\Gamma_{0}$ is non-empty.
Moreover, we have

\vspace*{5pt}
\begin{lemma}\label{decomp}
If $\Gamma$ is a Thurston obstruction, then $\Gamma_0$ is also a
Thurston obstruction. In particular, under a permutation of the basis, we can write
\begin{equation}~\label{form}
A_\Gamma=\begin{pmatrix}
                 A_{\Gamma_\infty} & 0 \\
                 \bigstar& A_{\Gamma_0}
           \end{pmatrix}
\end{equation}
where $\lambda(A_{\Gamma_\infty})<1$ and
$\lambda(A_\Gamma)=\lambda(A_{\Gamma_0})\geq 1$.
\end{lemma}

\begin{proof} First, for every curve $\gamma\in \Gamma_0$, there exists an integer $k\geq 0$ and an
element $\gamma_{ob}\in \Gamma_{Ob}$ such that $\gamma$ is homotopic to a
component of $f^{-k}(\gamma_{ob})$. It follows that any non-peripheral
component $\widetilde{\gamma}$ of $f^{-1}(\gamma)$ is homotopic to a
component of $f^{-(k+1)}(\gamma_{ob})$.
Since $\Gamma$ is $f$-stable,
then there exists an element $\gamma_i\in \Gamma$ which is homotopic
to $\widetilde{\gamma}$. Therefore, any non-peripheral
component of $f^{-1}(\gamma)$ is homotopic to an element
$\gamma_i\in \Gamma$ whose depth is at most $k+1$. This implies that $\gamma\in \Gamma_0$.
Thus $\Gamma_0$ is $f$-stable.

Let us write $\Gamma_{\infty}=\{\gamma_{1}, \cdots, \gamma_{s}\}$. Then $\Gamma_{0}=\{ \gamma_{s+1}, \cdots, \gamma_{n}\}$. Since $\Gamma_{0}$ is $f$-stable, $A_\Gamma$ must be of the form of (\ref{form}).
Furthermore, since $\Gamma_{Ob}\subset \Gamma_0$, we have that
$$
\lambda(A_{\Gamma_\infty})< 1 \ \ \text{and}\ \
\lambda(A_{\Gamma_0})=\lambda({A_\Gamma})\geq 1.
$$
\end{proof}

Now we study the associated matrix $A$ for a sub-hyperbolic
semi-rational branched covering $f$. For each disk $D_i$ in $\Lambda$,
we take a point $b_i$ on the boundary $\partial D_i$. Set
\begin{equation}~\label{defne}
E=P_1\cup \cup_i\{a_i, b_i\}.
\end{equation}
Let $p=\sharp E$. It is obvious that every multi-curve
$\Gamma $ in $\widehat{\mathbb C}\setminus Q$ is a multi-curve in
$\widehat{\mathbb C}\setminus E$. It follows that there are only finite number of
possible matrices for all linear transformations $f_{\Gamma}$ (refer to~\cite[Lemma 1.2]{DH}).
(There are infinitely many possible multi-curves $\Gamma$.)
Therefore, we have that

\vspace*{5pt}
\begin{proposition}~\label{beta}
There is a number $0< \beta\leq 1$ depending only on the degree $d$ of $f$
and the cardinality $p$ of $E$ such that for any irreducible multi-curve
$\Gamma$ in $\widehat{\mathbb{C}}\setminus Q$ (not necessarily
$f$-stable) with $\lambda(A_{\Gamma})\geq 1$, let ${\bf v}$ be the
unique positive eigenvector of $A_\Gamma$ corresponding to
$\lambda(A_{\Gamma})\geq 1$ with $\|{\bf v}\|=1$, then
the smallest coordinate of ${\bf v}$ is bounded below by $\beta$.
\end{proposition}

\begin{proof} Since there are only finitely many possible matrices for all irreducible multi-curves,
there are finitely many simple, positive, maximal eigenvalues. Thus there are finitely many
positive eigenvectors ${\bf v}$ with $\|{\bf v}\|=1$. This gives the proposition.
\end{proof}

\vspace*{5pt}
\begin{proposition}\label{one half}
There exists a positive integer $m$ such that for any non-empty $f$-stable multi-curve $\Gamma$, if it is a Thurston obstruction,
$$
||A^m_{\Gamma_\infty}||<1/2.
$$
\end{proposition}

\begin{proof}
Since there are only finitely many matrices $A_{\Gamma}$ corresponding to all $\Gamma$, there are only finitely many
$A_{\Gamma_{\infty}}$. For each $A_{\Gamma_{\infty}}$, $\lambda (A_{\Gamma_{\infty}})<1$. Thus we have an integer $m>0$ such that
$$
||A^m_{\Gamma_\infty}||<1/2.
$$
\end{proof}

Every multi-curve $\Gamma$ can contain at most $p-3$ curves, so we have that

\vspace*{5pt}
\begin{proposition}~\label{upperbound}
There is a positive integer $M$ depending on $p$ such that for any
$f$-stable multi-curve $\Gamma$ in $\widehat{\mathbb{C}}\setminus Q$, the depth of
every $\gamma\in \Gamma_{0}$ is less than or equal to $M$.
\end{proposition}

\section{Teichm\"uller space and short geodesics.}

Suppose $f$ is a sub-hyperbolic semi-rational branched covering.
Recall $Q$ and $P_{1}$ defined in (\ref{p1}) and (\ref{defq}) and
the assumption that $0,\ 1,\ \infty \in P_1$.
Let ${\mathcal M}(\mathbb{C})$ be the unit ball of the space $L^{\infty} (\mathbb{C})$. That is,
it is the set of all measurable functions $\mu$ on
$\mathbb{C}$ such that essential supremum norm $\|\mu\|_{\infty} <1$. Each element $\mu\in {\mathcal M}(\mathbb{C})$
is called a Beltrami coefficient since the measurable Riemann mapping theorem~\cite{AB} says that the Beltrami equation
$$
\phi_{\overline z} =\mu \phi_{z}
$$
has a unique quasiconformal self-map $\phi^{\mu}$ of $\widehat{\mathbb C}$ fixing $0$, $1$, and $\infty$ as a solution,
which depends on $\mu\in {\mathcal M}({\mathbb C})$ holomorphically. The map $\phi^{\mu}$ is called the normalized
solution.

\vspace*{5pt}
\begin{definition} The Teichm\"uller space $T_f$ is the equivalence class $[\mu]$
for $\mu \in {\mathcal M}({\mathbb C})$ satisfying that $\mu|Q=0$ a.e., where
$\mu_1$ and $\mu_2$ are equivalent if and
only if $\phi^{\mu_1}$ is isotopic to $\phi^{\mu_2}$ rel $Q$.
Furthermore, we can define the Teichm\"uller distance between two points
$x=[\mu]$ and $y=[\nu]$ in $T_f$ as
$$
\text{d}_T(x,y)=\frac{1}{2}\text{min}_{\widetilde{\mu}\in[\mu],\widetilde{\nu}\in[\nu]}\log K[\phi^{\widetilde{\mu}}\circ (\phi^{\widetilde{\nu}})^{-1}]
$$
where $K[\phi]$ is the maximal dilation of the quasiconformal map $\phi$.
\end{definition}

From~\cite{Li} (or refer to~\cite{JZ}),
we knew that $T_{f}$ is the Teichm\"uller space
$T(\widehat{\mathbb C}\setminus Q)$ of Riemann surface
$\widehat{\mathbb C}\setminus Q$ with
boundary $\partial Q$. It is a complex manifold and
the projective map
$$
\Phi: {\mathcal M}({\mathbb C}) \rightarrow T_{f}$$
is a holomorphic split submersion.

Define the self-map $\sigma_f$ of the Teichm\"uller $T_f$ by
$$
\sigma_f([\mu])=[f^*(\mu)].
$$
In formula,
$$
(f^{*}\mu)(z) = \frac{\mu_f(z) + \mu(f(z)) \theta(z)}{1 +
\overline{\mu_f (z)} \mu(f(z)) \theta(z)},
$$
where $\theta(z)$ = $\overline{f_{z}}/f_{z}$ and $\mu_{f}(z) = f_{\bar{z}} /f_{z}$,
is the pull-back of $\mu$ by $f$.
Since
$$
\sigma_f = \Phi \circ f^{*}\circ \Phi^{-1}
$$
where $\Phi^{-1}$ means a local holomorphic section of $\Phi$. Thus
$$
\sigma_f: T_{f}\rightarrow T_{f}
$$
is a holomorphic map. Since the Teichm\"uller metric $\text{d}_T$
coincides with the Kobayashi metric on the complex manifold $T_{f}$ and
$\sigma_{f}$ is holomorphic, we have that
$$
\text{d}_T (\sigma_{f}(x), \sigma_{f}(y)) \leq \text{d}_T (x, y), \quad \forall x, y\in T_{f}.
$$
From~\cite{JZ}, we also know that
\begin{equation}~\label{contraction}
\text{d}_T (\sigma_{f}(x), \sigma_{f}(y)) < \text{d}_T (x, y), \quad \forall x, y\in T_{f}.
\end{equation}
We need more definitions and lemmas from~\cite{JZ} as follows.

Let $Z$ be a subset of $Q$ with $\sharp (Z)\geq 4$. Let $x=[\mu]\in
T_f$ and let $\gamma\in \widehat{\mathbb{C}}\setminus Z$ be a simple
closed and non-peripheral curve. We use $l_{Z}(\gamma, x)$ to denote
the hyperbolic length of the unique simple closed geodesic which is
homotopic to $\phi^\mu(\gamma)$ in the hyperbolic Riemann surface
$\widehat{\mathbb{C}}\setminus \phi^\mu(Z)$. We say $\gamma$ is a
$(\mu,Z)$-simple closed geodesic if $\phi^{\mu}(\gamma)$ is a simple
closed geodesic in $\widehat{\mathbb{C}}\setminus \phi^\mu(Z)$.

\vspace*{5pt}
\begin{remark}
From the definition of the Teichm\"uller space $T_f$, we know that the
definition of $l_{Z} (\gamma, x)$ is independent of the choice of $\mu$ in $x$.
\end{remark}

For $x_{0}\in T_{f}$, let $x_n=\sigma^n_f(x_0)$, $n=1, \cdots$ be a sequence in $T_{f}$.
Recall our definition of $E$ in (\ref{defne}).

\vspace*{5pt}
\begin{lemma}~\label{converges}
If there is a real number $a>0$ such that there is a point $x_{0}\in T_f$ and every
$(x_{n},E)$-simple closed geodesic $\gamma\subset \widehat{\mathbb C}\setminus Q$
has hyperbolic length greater than or equal to $a$, then the sequence $\{x_{n}\}_{n=0}^{\infty}$
is convergent in $T_f$ and the limiting point is the fixed point of $\sigma_f$ in $T_f$.
\end{lemma}

\begin{remark}
This lemma implies that if there exists an $x_0\in T_f$ such that the length of the shortest geodesics on
all the $x_{n}$ has a uniform lower bound, then $f$ has no
Thurston obstructions.
\end{remark}

\vspace*{5pt}
\begin{lemma}\label{comparision}
There exists an $\eta>0$ such that for any point $x=[\mu]\in T_f$
with $\mu(z)=0$ on $\cup_i A_i$ and for any $(x, E)$-simple geodesic
$\gamma \subset \widehat{\mathbb C}\setminus E$ with
$l_{E}(\gamma,x)<\eta$, we have $\gamma\subset \widehat{\mathbb C}
\setminus Q$. Moreover, for any $\epsilon >0$, there exists a $\delta>0$ such that
$$l_{E}(\gamma, x)>(1-\epsilon) l_{Q}(\gamma, x)$$ whenever
$l_{E}(\gamma,x)<\delta$.
\end{lemma}

\vspace*{5pt}
\begin{remark}~\label{delta0}
The above lemma implies that for any $x=[\mu]\in T_{f}$ with $\mu(z)=0$ for all
$z\in\cup_i A_i$, a sufficiently short geodesics in
$\widehat{\mathbb{C}}\setminus \phi^\mu(E)$ are homotopic to the
sufficiently short geodesics in $\widehat{\mathbb{C}}\setminus
\phi^\mu(Q)$. More precisely, we can find a constant $\delta_{0}>0$ such that
$$
\frac{1}{e} l_{Q}(\gamma, x)<l_{E}(\gamma, x)<l_{Q}(\gamma, x) \quad \hbox{whenever $l_{E}(\gamma,x)<\delta_{0}$}.
$$
\end{remark}

Suppose $x=[\mu]\in T_{f}$ and $Z\subset Q$. Define
$$
w_{Z}(\gamma,x)=-\log l_{Z}(\gamma,x).
$$
Consider the set
$$
L_{Z,x}=\{w_{Z}(\gamma,x)\}
$$
where $\gamma$ ranges over all the non-peripheral simple closed
curves in $ \widehat{\mathbb{C}} \setminus Q$.
Define
$$
w_Z(x)=\text{sup} \{w_Z(\gamma, x)\}
$$
and
$$
w_Z(\Gamma, x)=\text{max}_{\gamma\in \Gamma}{w_Z(\gamma,x)}.
$$
The following lemma is a general result for hyperbolic Riemann surfaces (refer to~\cite{DH,JZ}).
We just state it in our case.

\vspace*{5pt}
\begin{lemma}\label{lipschitz} Let $Z\subset Q$ be a finite subset with $\sharp Z\geq 4$ and $\gamma \subset  \widehat{\mathbb{C}}\setminus Q$ be
a non-peripheral simple closed curve. Then the function
$$x\mapsto w_Z(\gamma, x): T_f\to \mathbb{R}$$ is Lipschitz
with Lipschitz constant 2.
\end{lemma} Let
$$
A=\text{max}\{-\log\log(2\sqrt{2}+3), -\log \delta_0\}
$$
where $\delta_0$ is the number in Remark~\ref{delta0}.
Note that $\log (2\sqrt{2}+3)$ is the magic number in the theory of hyperbolic Riemann surfaces such that
for any hyperbolic Riemann surface $S$, any two simple closed geodesics $\gamma$ and $\gamma'$ in $S$ are disjoint
whenever the hyperbolic lengths of $\gamma$ and $\gamma'$ are less than $\log (2\sqrt{2}+3)$.
This implies that for any point $x\in T_{f}$, there are at most $p-3$
curves $\gamma$ with $l_{E}(\gamma, x)\leq \log(2\sqrt{2}+3)$.

For any $J>0$, let $(a,b)$ be the lowest interval in
$\mathbb{R}\setminus L_{E,x}$ such that $a\geq A$ and $b-a=J$. For
any $x=[\nu]\in T_{f}$, define
$$
\Gamma_{J,x}=\{\gamma\ | \gamma \text{ is a simple closed geodesic on }
R_x\  \text{and} \ w_{E}(\gamma,x)\geq b\}.
$$
Then $\Gamma_{J,x}$ is a multi-curve consisting of the geodesics which are
sufficiently short on $\mathbb{\widehat{C}}\setminus \phi^\mu(E)$. This is equivalent saying that
they are all the simple closed curves in $\widehat{\mathbb{C}}\setminus
\phi^{\mu}(Q)$ which are homotopic to
sufficiently short simply closed geodesics on $\widehat{\mathbb{C}}\setminus
\phi^{\mu}(Q)$.
There are at most $p-3$ elements in $\Gamma_{J,x}$ for any $x$ and they are pairwise disjoint.

For any $x\in T_{f}$, let $D=d_T(x,\sigma_f(x))$.

\vspace*{5pt}
\begin{lemma}~\label{stable}
If $J\geq \log d+2D+1$ and $\Gamma_{J,x}\not=\emptyset$, then $\Gamma_{J,x}$ is an $f$-stable multi-curve.
\end{lemma}
See Lemma 7.3 in~\cite{JZ}.

\section{Upper bound for $\Gamma_\infty$}

We still keep the notations in the previous sections.
Suppose $x_{0}\in T_{f}$ and $x_{n}= \sigma_f^{n}(x_{0})$ for all
$n\geq 1$. Then we have a sequence $\{ x_{n}\}_{n=0}^{\infty}$ in
$T_{f}$.

For all $n>0$ and all $z\in \cup_{i}A_i$, we have that
$\mu_n(z)=0$, where $[\mu_n]=x_n$, since $f(\cup_i A_i)\subset
\cup_i D_i$ as we constructed $\{A_{i}\}$ as the shielding rings.

Recall the definition of $E=P_1\cup \cup_i\{a_i, b_i\}$ in (\ref{defne}) and $m$ in Proposition \ref{one half}.
Let $$
P_2=E\cup f^m(E)\cup \cup_{1\leq j\leq m}f^{j}(\Omega_f)\subset Q.
$$
The following lemma is also from~\cite{JZ}.

\vspace*{5pt}
\begin{lemma}~\label{eps}
There exists an $\epsilon_0>0$, such that for any $x=[\mu]\in T_f$ with $\mu(z)=0$ for all $z\in \cup_i A_i$,
and for any $(\mu, P_2)$-simple closed geodesic $\gamma'$, if $l_{P_2}(\gamma',x)<\epsilon_0$,
then there is a $(\mu, E)$-simple closed geodesic $\gamma$ such that $\gamma'$ is homotopic to $\gamma$ in $\widehat{\mathbb{C}}\setminus P_2$.
\end{lemma}

The following lemma is also a general result in the theory of hyperbolic Riemann surfaces and the reader can find a proof in~\cite{DH}.

\vspace*{5pt}
\begin{lemma}~\label{lip}
Let $X$ be a hyperbolic Riemann surface, $P\subset X$
is a finite subset, and $\sharp P<p$. Let $X'=X\setminus P$ and
$L<\log(3+2\sqrt 2)$. Let $\gamma$ be a simple closed geodesic on
$X$, and let $\gamma_1',\cdots, \gamma_k'$ be all the geodesics on $X'$
homotopic to $\gamma$ in $X$ whose hyperbolic length on $X'$ is less than $L$. Set
$l=l_X(\gamma)$ and $l_i'=l_{X'}(\gamma_i')$. Then:
\begin{itemize}
\item[(1)] $k\leq p+1$;
\item[(2)] for all $i$, $l_i'\geq l$;
\item[(3)] $\frac{1}{l}-\frac{1}{\pi}-\frac{(p+1)}{L}<\sum_{i=1}^k\frac{1}{l_i'}<\frac{1}{l}+\frac{(p+1)}{\pi}.$
\end{itemize}
\end{lemma}

The next proposition is essential for our proof.

\vspace*{5pt}
\begin{proposition}~\label{fundamental}
Let $m$ be the constant in Proposition \ref{one half}. Let $x_0\in T_f$ and $x_n=\sigma_f^n(x_0)$ for $n>0$. There exists
a constant $C(J)>0$ depending on $p$, $d$, $\epsilon_0$, $D=d_{T}(x_0,x_1)$
and $J\geq m(\log d+2D+1)$ such that if $w_E(x_{0})>C(J)$,
then $\Gamma=\Gamma_{J,x_{0}}\neq \emptyset$ is a stable multi-curve.
Moreover, if $\Gamma_\infty\neq
\emptyset$, then
$$
w_E(\Gamma_\infty, x_{m})\leq w_E(\Gamma_\infty, x_{0}).
$$
\end{proposition}

\begin{proof}

If $w_E(x_0)\geq A+(p-3)J$, then $\Gamma_{J,x_0}$ is non-empty, since $R_{x_0}$ has at most $(p-3)$ simple closed geodesics with hyperbolic length less than $e^{-A}$ (they are not homotopic to each other). From Lemma~\ref{stable}, $\Gamma=\Gamma_{J,x_0}$ is also $f$-stable.

Suppose $\Gamma_\infty\neq \emptyset$ and $A_\Gamma$ is in the form of (\ref{form}).  From Proposition~\ref{one half}, $||A^m_{\Gamma_\infty}||<1/2$.

For each $\gamma_j\in \Gamma_{J, x_0}$, let $\gamma_{i,j,\alpha}$ be any component of $f^{-m}(\gamma_j)$ homotopic to $\gamma_i$ in $\widehat{\mathbb{C}}\setminus Q$. Then $\gamma_{i,j,\alpha}$ is also homotopic to $\gamma_i$ in $\widehat{\mathbb{C}}\setminus E$. Let $g=\phi^\mu\circ f^m\circ (\phi^{\nu})^{-1}$, where
$[\mu]=x_{0}$ and $[\nu]=x_m$. Then $g$ is a rational map and
$$
g: \widehat{\mathbb{C}}\setminus \phi^{\nu}(f^{-m}(P_2))\rightarrow \widehat{\mathbb{C}}\setminus \phi^\mu(P_2)
$$ is a holomorphic covering map. Therefore
$$
l_{f^{-m}(P_2)}(\gamma_{i,j,\alpha},x_m)=d_{i,j,\alpha}l_{P_2}(\gamma_{j},x_0),
$$
where $d_{i,j,\alpha}$ is the degree of $f^m: \gamma_{i,j,\alpha}\rightarrow \gamma_j$.
We get
$$
\sum_\alpha \frac{1}{l_{f^{-m}(P_2)}(\gamma_{i,j,\alpha},x_m)}=\Big( \sum_\alpha\frac{1}{d_{i,j,\alpha}}\Big) \frac{1}{l_{P_2}(\gamma_{j},x_0)}=b_{ij}\frac{1}{l_{P_2}(\gamma_{j},x_0)},
$$
where $b_{ij}$ is the $ij$-entry of $A_\Gamma^m$.

Since $E\subset P_2$, the inclusion
$$
\iota: \widehat{\mathbb C}\setminus P_{2}\hookrightarrow \widehat{\mathbb C}\setminus E
$$
decreases the hyperbolic distances. So we have that $l_{P_2}(\gamma_{j},x_0)>l_{E}(\gamma_{j},x_0)$ for any $\gamma_j$. It follows that
$$
\sum_\alpha \frac{1}{l_{f^{-m}(P_2)}(\gamma_{i,j,\alpha},x_m)} < b_{ij}\frac{1}{l_{E}(\gamma_{j},x_0)}.
$$

From the definitions of $P_{2}$ and $E$, we know that $E\subset f^{-m}(P_2)$. Let $C=C(d,m,p)=\sharp (f^{-m}(P_2)\setminus E)$, where $p=\sharp E$.

We claim that for any $(\nu, f^{-m}(P_2))$-simple closed geodesic $\gamma$ which is homotopic to $\gamma_i$ in $\widehat{\mathbb{C}}\setminus E$, either $\gamma$ is homotopic to some $\gamma_{i,j,\alpha}$ in $\widehat{\mathbb{C}}\setminus f^{-m}(P_2)$ or
$$
l_{f^{-m}(P_2)}(\gamma,x_m)>\text{min}\{e^{-(A+PJ)}, \epsilon_0\},
$$ where $\epsilon_0$ is the constant in Lemma~\ref{eps}.

We prove the claim. In fact, if $\gamma$ is not homotopic in $\widehat{\mathbb{C}}\setminus f^{-m}(P_2)$ to some $\gamma_{i,j,\alpha}$, then $f^m(\gamma)$ is a $(\mu,P_2)$-simple closed geodesic which is not homotopic to any $\gamma_j$ in $\widehat{\mathbb{C}}\setminus P_2$. Then there are two cases: either (1) $f^m(\gamma)$ is homotopic in $\widehat{\mathbb{C}}\setminus P_2$ to some $(\mu, E)$-simple closed geodesic $\xi$ which does not belong to $\Gamma_{J,x_0}$, then we have
$$
l_{P_2}(f^m(\gamma),x_0)>l_E(f^m(\gamma),x_0)=l_E(\xi,x_0)>e^{-a}>e^{-(A+PJ)}
$$
or (2) $f^m(\gamma)$ is not homotopic in $\widehat{\mathbb{C}}\setminus P_2$ to any $(\mu,E)$-simple closed geodesic, then by Lemma~\ref{eps}, we have $$
l_{P_2}(f^m(\gamma), x_0)>\epsilon_0.
$$
Thus we have
$$
l_{f^{-m}(P_2)}(\gamma,x_m)\geq l_{P_2}(f^m(\gamma), x_0)>\text{min}\{e^{-(A+PJ)}, \epsilon_0\}.
$$
This proves the claim.

From the left hand of the inequality given by (3) in Lemma~\ref{lip}, for each $\gamma_i\in \Gamma$, we have
$$
\frac{1}{l_E(\gamma_i,x_m)}-\frac{1}{\pi}-\frac{C+1}{\text{min} \{ e^{-(A+PJ)}, \epsilon_0\} }\leq \sum_{j,\alpha} \frac{1}{l_{f^{-m}(P_2)}(\gamma_{i,j,\alpha}, x_m)}\leq \sum_j b_{ij} \frac{1}{l_{E}(\gamma_{j},x_0)}.
$$
Suppose $\Gamma_\infty=\{\gamma_1,\ \cdots,\ \gamma_s\}\subset \Gamma$. Then for each $\gamma_{i}\in\Gamma_\infty$, from the form (\ref{form}) of $A_{\Gamma}$,
$$
\frac{1}{l_E(\gamma_i, x_m)}\leq \sum_{j=1}^s b_{ij}\frac{1}{l_E(\gamma_j, x_0)} +\frac{1}{\pi}+\frac{C+1} {\text{min}\{e^{-(A+PJ)}, \epsilon_0\}}.
$$
Let
$$
{\bf v}_1=\begin{pmatrix} \frac{1}{l_E(\gamma_1, x_m)} \\ \vdots \\ \frac{1}{l_E(\gamma_s, x_m)} \end{pmatrix} \ \text{and}\  {\bf v}=\begin{pmatrix}\frac{1}{l_E(\gamma_1,x_0)} \\ \vdots \\ \frac{1}{l_E(\gamma_s,x_0)}\end{pmatrix}.
$$
Since $\|A^m_\infty\|<1/2$,
$$
\|{ \bf v}_1\|< \frac{1}{2}\|{ \bf v}\|+\frac{1}{\pi}+\frac{C+1} {\text{min}\{e^{-(A+PJ)}, \epsilon_0\}}.
$$
Define
$$
C(J)=\text{max}\Big\{ 2\Big(\frac{1}{\pi}+\frac{C+1}{\text{min}\{{e^{-(A+PJ)}, \epsilon_0}\}}\Big),A+(p-3)J\Big\}.
$$
If $w_E(\Gamma_\infty, x_0)\geq C(J)$, then we have
$$
w_E(\Gamma_\infty, x_m)< w_E(\Gamma_\infty, x_0).
$$
\end{proof}

\vspace*{5pt}
\begin{lemma}\label{lbcurve}
Let $J\geq m(\log d+2D+1)$. Suppose $w_E(x_{0})<C(J)$ and suppose
$\Gamma=\Gamma_{J,x_{k}}\neq \emptyset$ for some $k\geq 0$. Let
$E(J)=C(J)+2mD$. If $\Gamma_\infty\neq \emptyset$, then for all $n$,
$$
w_E(\Gamma_\infty, x_{n})<E(J).
$$
Moreover, if $w_E(\gamma,x_{k})\geq E(J)$, then $\gamma\in
\Gamma_0$.
\end{lemma}

\begin{proof}
We prove the first inequality by contradiction. Suppose there is an
$n>0$ such that $w_E(\Gamma_\infty, x_{n})\geq C(J)+2mD$. Suppose
$n_0$ is the first integer having this property. Then we have
$w_E(\Gamma_\infty, x_{n_{0}-m})\geq C(J)$. Then by Proposition~\ref{fundamental}
and the fact that $n_0$ is the first integer such
that $w_E(\Gamma_\infty, x_{n_{0}})\geq C(J)+2mD$, we have
$$
w_E(\Gamma_\infty, x_{n_{0}})\leq w_E(\Gamma_\infty,
x_{n_{0}-m})<C(J)+2mD.
$$
This is a contradiction.

If $w_E(\gamma,x_{k})\geq E(J)>C(J)\geq A+(p-3)J$, then $\gamma\in
\Gamma_{J,x_{k}}=\Gamma$ since there are at most $p-3$ simple closed
curves in $R_{x_k}$ such that $w_E(\gamma, x_k)>A$. But $\gamma\notin\Gamma_\infty$ because of the first
conclusion and the assumption. Therefore, $\gamma\in \Gamma_0$.
\end{proof}

\section{Lower bound for $\Gamma_0$}

In order to get the lower bound for $\Gamma_0$, we need the following
definition.

\vspace*{5pt}
\begin{definition}
Let $\kappa$ be a real number. A sequence $\{a_n \}_{n=0}^\infty$ of
real numbers is called $\kappa$-quasi-nondecreasing if for all
$n_1<n_2$ we have $a_{n_{2}}-a_{n_{1}}\geq \kappa$. A sequence is
called quasi-nondecreasing if it is $\kappa$-quasi-nondecreasing for
some $\kappa$.
\end{definition}

It is easy to check that the following two properties are true.

\vspace*{5pt}
\begin{property}~\label{pr1}
Suppose $\{a_n\}_{n=0}^\infty$ and  $\{b_n\}_{n=0}^\infty$ are two
sequences. If $\{a_n\}_{n=0}^\infty$ is $\kappa$-quasi-nondecreasing
and if $|a_n-b_n|<r$ for all $n$, then $\{b_n\}$ is
$(\kappa-2r)$-quasi-nondecreasing.
\end{property}

\vspace*{5pt}
\begin{property}~\label{pr2}
Suppose $\{a_n\}$ is quasi-nondecreasing and unbounded. Then
$a_n\rightarrow +\infty$ as $n\rightarrow +\infty$.
\end{property}

Recall that any $x=[\mu]\in T_f$ represents a complex structure on $ \widehat{\mathbb{C}} \setminus Q$, which makes
$ \widehat{\mathbb{C}}\setminus Q$ a hyperbolic Riemann surface $R_x$. For any simple closed geodesic
$\gamma$ on $R_x$, let $A(\gamma, x)$ be the Riemann surface, conformally
isomorphic to an annulus, obtained by taking the unit disk ${\mathbb D}$
modulo a $\mathbb{Z}$-subgroup of the fundamental group of $R_x$
generated by $\gamma$.  It is a covering space of $R_{x}$. The core curve of $A(\gamma, x)$ is a
geodesic of length $l_{Q}(\gamma,x)$ and
\begin{equation}~\label{ann}
\text{mod}(A(\gamma, x))=\frac{\pi}{l_{Q}(\gamma,x)},
\end{equation}
where $\text{mod}(A)$ means the modulus of an annulus $A$.

If $\gamma$ is a simple closed geodesic of hyperbolic length $l$ on the Riemann
surface $R_x$, then there is an embedding annulus $a(\gamma, x)$ of
modulus $m(l)$ which is continuous and decreasing and satisfies
$$
\frac{\pi}{l}-1<m(l)<\frac{\pi}{l}.
$$
Thus for all $x\in T_f$, we have
\begin{equation}~\label{modinq}
\text{mod}(A(\gamma, x))-1<\text{mod}(a(\gamma,
x))<\text{mod}(A(\gamma,x)).
\end{equation}

We need the following technical lemma:

\vspace*{5pt}
\begin{lemma}~\label{tech}
If $t\geq 1$, then $\log(t+1)-1<\log t$.
\end{lemma}

\begin{proof}
For $t\geq 1$,
$$
\log(t+1)-\log t=\log(\frac{t+1}{t})\leq\log 2< 1.
$$
\end{proof}

If $w_Q(\gamma, x)\geq \log \frac{2}{\pi}=-0.451582705\cdots$, then
we have $\text{mod}(A(\gamma,x))-1\geq 1$. By taking logarithms on all
terms of Inequality (\ref{modinq}) and by applying Lemma~\ref{tech}
and Equation~(\ref{ann}), we have
$$
\log\pi-1+w_Q(\gamma,x)<\log\text{mod}(a(\gamma,x))<\log\pi+w_Q(\gamma,x).
$$
It follows that, if $w_Q(\gamma, x)\geq \log \frac{2}{\pi}$, then
\begin{equation}~\label{modinq2}
|\log\text{mod}(a(\gamma,x))-w_Q(\gamma,x)|<\log\pi.
\end{equation}

Given a multi-curve $\Gamma$, we denote vectors of moduli
$(\text{mod}(A(\gamma, x)))$ and $(\text{mod}(a(\gamma, x)))$ by
$\text{mod}(A(\Gamma, x))$ and $\text{mod}(a(\Gamma,x))$
respectively. Define
$$
\underline{\text{mod}}(A(\Gamma,x))=\text{min}_{\gamma\in
\Gamma}\{\text{mod}(A(\gamma,x))\}
$$
and
$$
\underline{\text{mod}}(a(\Gamma,x))=\text{min}_{\gamma\in
\Gamma}\{\text{mod}(a(\gamma,x))\}.
$$

\vspace*{5pt}
\begin{lemma}~\label{min}
Let $\beta$ be the constant in Proposition~\ref{beta}.
Let $\Gamma$ be an irreducible multi-curve. Suppose
the leading eigenvalue of the matrix
$A_\Gamma$ is greater than or equal to $1$. Then for any $x_{0} \in T_f$ and
$x_n=\sigma_f^n(x_{0})$, $n>0$,
\begin{itemize}
\item[(1)] $\underline{\text{mod}}(A(\Gamma,x_n))\geq
\beta\underline{\text{mod}}(a(\Gamma,x_{0}))$ and
\item[(2)] $\underline{\text{mod}}(a(\Gamma,x_n))\geq
\beta\underline{\text{mod}}(a(\Gamma,x_{0}))-1$.
\end{itemize}
\end{lemma}

\begin{proof}
Since for any $n$, $f^n:  \widehat{\mathbb{C}}\rightarrow  \widehat{\mathbb{C}}$ is a branched covering, we can similarly define the linear map $f^n_\Gamma: \mathbb{R}^\Gamma\rightarrow \mathbb{R}^\Gamma$.
Let $B$ be the corresponding matrix for the linear map $f^n_\Gamma$ with the basis $\Gamma$. It is easy to see that $B\geq A_\Gamma^n$.

Let $\bf{v}$ be the unique positive eigenvector of $A_{\Gamma}$ with $\| \bf{v}\|=1$. Let $\bf{1}$ denote the vector whose coordinates are all equal to $1$.
Then
$$\text{mod}(a(\Gamma,x_{0}))\geq \underline{\text{mod}}(a(\Gamma,x_{0}){\bf 1} \geq \underline{\text{mod}}(a(\Gamma,x_{0})){\bf v}. $$
For any $n\geq 1$, let $\gamma^n_{i,j,\alpha}$ be the components of
$f^{-n}(\gamma_j)$ homotopic to $\gamma_i$, and $a^n_{i,j,\alpha}$
be the components of $f^{-n}(a(\gamma_j,x_{0}))$ homotopic to $\gamma_i$.
Then
$$
\text{mod}(a^n_{i,j,\alpha})=\text{mod}(a(\gamma_j,x_{0}))
/d^n_{i,j,\alpha},
$$ where
$d^n_{i,j,\alpha}=\text{deg}f^n|_{\gamma^n_{i,j,\alpha}}$. Since
$a^n_{i,j,\alpha}$ are disjoint annuli homotopic to
the curve $\gamma_i$,
we have
$$
\sum_{\alpha,j} \text{mod}(a^n_{i,j,\alpha})\leq
\text{mod}(A(\gamma_i,x_n)).
$$
(One can obtain this inequality by lifting them to the covering space $A(\gamma_i,x_n)$ of $R_{x_n}$
and then by using Gr\"otzsch's inequality.)
Consequently we get
\begin{eqnarray*}
\text{mod}(A(\Gamma,x_n))&\geq& \text{mod} (a(\Gamma, x_{n})) \geq B\text{mod}(a(\Gamma,x_{0})) \\
                         &\geq& A_\Gamma^n\text{mod}(a(\Gamma,x_{0})) \geq A_\Gamma^n\underline{\text{mod}}(a(\Gamma,x_{0}))\bf{v}\\
                         &\geq& \underline{\text{mod}}(a(\Gamma,x_{0})){\bf{v}} \geq \beta \underline{\text{mod}}(a(\Gamma,x_{0}))\bf{1}
\end{eqnarray*}
Hence for all $\gamma\in \Gamma$, we have
$\text{mod}(A(\gamma,x_n))\geq \beta
\underline{\text{mod}}(a(\Gamma,x_{0}))$.
The second conclusion follows the first one and Inequality (\ref{modinq}).
\end{proof}

\vspace*{5pt}
\begin{lemma}~\label{tech2} If $a,\ b>0$, $\beta>0$, and $e^a\geq \beta e^ b -1$,
then $a-b\geq \log\beta-1$. \end{lemma}

\begin{proof} If $\beta\exp b -1\geq 1$, then by Lemma~\ref{tech}, we have $$\log
(\beta\exp b -1)\geq \log(\beta e^b)-1\geq \log\beta+b-1.$$ Hence by
the assumption, we have $a-b\geq \log\beta-1$.

If $\beta\exp b -1<1$,
then $b<\log 2-\log\beta$. Since $a>0$,
$$
a-b>0-b=-b>\log\beta-\log 2>\log\beta-1.
$$
\end{proof}

For any $x\in T_f$ and any multi-curve $\Gamma$, define
$$
\underline{w}(\Gamma,x)=\text{min}_{\gamma\in
\Gamma}w_Q(\gamma,x).
$$

\vspace*{5pt}
\begin{lemma}~\label{gamma}
Suppose $\Gamma$ is an irreducible multi-curve and suppose
the leading eigenvalue of the matrix $A_{\Gamma}$ is greater than or equal to $1$.
For any $x_{0}\in T_{f}$, if $\underline{w}(\Gamma,x_{0})\geq
\log(3/\beta)+\log\pi$, then the sequence $\{\underline{w}(\Gamma,x_n)\}_{n=0}^{\infty}$, where $x_n=\sigma_f^n(x_{0})$,
is $(\log\beta-1-2\log\pi)$-quasi-nondecreasing.
\end{lemma}

\begin{proof} For $\underline{w}(\Gamma,x_0)\geq
\log(3/\beta)+\log\pi>\log\frac{2}{\pi}$,
by Inequality (\ref{modinq2}), we have
$$
\log \underline{\text{mod}}(a(\Gamma,x_0))\geq \log(\frac{3}{\beta}).
$$
That is,
$\underline{\text{mod}}(a(\Gamma,x_0))\geq 3/\beta$. So
$\beta\underline{\text{mod}}(a(\Gamma,x_0))-1\geq 2$. By Lemma~\ref{min}, we
have that for all $n\geq 0$,
\begin{equation}~\label{modinq3}
\underline{\text{mod}}(a(\Gamma,x_n))\geq 2.
\end{equation}
Now consider the sequence $y_n=\log
\underline{\text{mod}}(a(\Gamma,x_n))$. Choose arbitrarily
$n_2>n_1\geq 0$,  and let $a=y_{n_2},\ b=y_{n_1}$ and $n=n_2-n_1$.
By Lemma~\ref{min}, we have $e^a\geq \beta e^b-1$.
 Applying Lemma~\ref{tech2}, we
have $a-b\geq \log\beta-1$, so the sequence $\{y_n\}$ is a
$(\log\beta -1)$-quasi-nondecreasing.

By Inequalities (\ref{modinq}) and (\ref{modinq3}),
we have $\underline{\text{mod}}(A(\Gamma,x))\geq 2$.
This implies that $\log \pi+\underline{w}(\Gamma,x_n)\geq
\log 2$. That is, $\underline{w}(\Gamma,x_n)\geq \log(2/\pi)$.
Since $\text{mod}(a(\gamma,x_n))$ is continuous and decreasing with $l_Q(\gamma,x_n)$,
we obtain
$$
\underline{\text{mod}}(a(\Gamma,x_n))=\text{mod}(a(\gamma, x_{n}))\quad \hbox{and}\quad
\underline{w}(\Gamma,x_n)=w_{Q} (\gamma, x_{n})
$$
at the same $\gamma\in \Gamma$. This further implies that
$$
|y_n-\underline{w}(\Gamma,x_n)|<\log\pi.
$$
From Property~\ref{pr1}, we finally have that $\underline{w}(\Gamma,x_n)$ is $(\log\beta-1-2\log\pi)$-quasi-nondecreasing.
\end{proof}

\vspace*{5pt}
\begin{lemma}~\label{modinq4}
Let $k\geq 1$ be an integer. For any $x_{0}\in T_{f}$, let $x_{n}=\sigma_{f}^{n}(x_{0})$ for $n>0$. Let $D=d_{T}(x_{0},x_1)$.
If $\gamma_1, \gamma_2$ are non-peripheral curves in
$ \widehat{\mathbb{C}}\setminus Q$ such that some component of $f^{-k}(\gamma_1) $ is
homotopic to $\gamma_2$, then
$$
w_{Q}(\gamma_2,x_{0})\geq w_{Q}(\gamma_1,x_{0})-k(\log d+2D).
$$
\end{lemma}

\begin{proof} Let $Y=f^{-k}(R_{x_{0}})$. Then $Y\subset R_{x_k}$ is a Riemann surface and
$f^k: Y\rightarrow R_{x_{0}}$ is a holomorphic covering map of degree $d^k$. Then
$$
l_{Y}(\gamma_2)\leq d^k l_Q(\gamma_1,x_{0}).
$$
Since the inclusion map
$\iota: Y\hookrightarrow R_{x_k}$ decreases the hyperbolic lengths,
$$
l_{Q}(\gamma_2,x_k)\leq d^kl_{Q}(\gamma_1,x_{0}).
$$
It follows that
$$
w_Q(\gamma_2,x_k)>w_Q(\gamma_1,x_{0})-k\log d.
$$
Since $\sigma_f$ decreases the Teichm\"uller distance $d_{T}$,
$$
d_T(x_i,x_{i+1})\leq d_T(x_{0},x_1)=D.
$$
The map $\gamma \mapsto w_Q(\gamma,
x)$ for any $x\in T_{f}$ is a Lipschitz function with Lipschitz constant $2$ (see Lemma~\ref{lipschitz}), so we have that
$$
w_Q(\gamma_2,x_{0})\geq w_Q(\gamma_2,x_k)-2kD\geq w_Q(\gamma_1,x_{0})-k(2D+\log d).
$$
\end{proof}

\vspace*{5pt}
\begin{lemma} Suppose $\Gamma$ is an irreducible multi-curve. Then for all
$\gamma_i,\ \gamma_j\in \Gamma$ and all $x\in T_{f}$,
$$
|w_{Q}(\gamma_i,x)-w_{Q}(\gamma_j,x)|\leq (p-3)(\log d+2D).
$$
\end{lemma}

\begin{proof} Since $\Gamma$ is irreducible, there is an integer $q\leq
\sharp\Gamma\leq p-3$ such that $\gamma_i$ is homotopic to a preimage of
$f^{-q}(\gamma_j)$. By Lemma~\ref{modinq4}, we see that $w_{Q}(\gamma_i, x)\geq
w_{Q}(\gamma_j, x)-(p-3)(\log d+2D)$. By exchanging $i$ and $j$, we complete the proof.
\end {proof}

\vspace*{5pt}
\begin{proposition}~\label{main0}
Suppose $\Gamma$ is an $f$-stable multi-curve satisfying $\Gamma=\Gamma_0$. Let
$x_{0}\in T_f$ and $x_{n}=\sigma_{f}^{n}(x_{0})$, $n>0$. Let $D=d_{T}(x_{0},x_{1})$.
Suppose $\text{min}_\gamma w_{Q}(\gamma,x_{0})\geq \log (3/\beta)+\log\pi$, where $\beta$ is
the number in Proposition~\ref{beta}. Write
$\Gamma=\Gamma'\sqcup\Gamma^{''}$, where $\Gamma'=\Gamma_{Ob}$ is the union of
the irreducible component $\Gamma_j$ of $\Gamma$ for which
$\lambda(A_{\Gamma_j})\geq 1$. Then
\begin{itemize}
\item[(1)] for all $\gamma\in \Gamma'$, $\{w_Q(\gamma, x_{n})\}_{n\geq 0}$ is
$\kappa$-quasi-nondecreasing, where $\kappa=\log \beta
-1-2\log\pi-2(p-3)(\log d+2D)$;
\item[(2)] for all $\gamma\in \Gamma^{''}$ and all $n\geq 0$,
$$
w_Q(\gamma, x_{n})\geq \text{min}_{\gamma'\in
\Gamma'}\{w_Q(\gamma',x_{n})\}-M(\log d+2D),
$$
where $M$ is the constant in Proposition~\ref{upperbound}.
\item[(3)] Suppose $\text{min}_{\gamma\in \Gamma} w_Q(\gamma,x_{0})\geq J_A-1$, where
$$
J_A=\text{max}\{\log(3/\beta)+\log \pi, A\}+\kappa+M(\log d+2D)+1.
$$
Then for all $\gamma\in \Gamma$ and for all $n\geq 0$, we have
$$
w_Q(\gamma,x_n)\geq A.
$$
\end{itemize}
\end{proposition}

\begin{proof}
Let $\Gamma_j$ be an irreducible
component of $\Gamma$ for which $\lambda(\Gamma_j)\geq 1$.
By the assumption that $\underline{w}(\Gamma,x_{0})\geq
\log(3/\beta)+\log\pi$, we have
$\{\underline{w}(\Gamma_j,x_n)\}$ is
$\log\beta-1-2\log\pi$-quasi-nondecreasing.

Since $\Gamma_j$ is an irreducible multi-curve, by Lemma~\ref{modinq4} and
Property~\ref{pr1}, we have for each
$\gamma\in \Gamma_j$, the sequence $\{w_Q(\gamma,x_k)\}_{k=0}^{\infty}$ is a
$\kappa=\log\beta-1-2\log\pi-2(p-3)(\log d+2D)$-quasi-nondecreasing.
This completes (1).

By Lemma~\ref{gamma} and (1), we have for all $\gamma\in
\Gamma^{''}$ and all $n\geq 0$,
$$
w_Q(\gamma, x_{n})\geq \text{min}_{\gamma'\in
\Gamma'}\{w_Q(\gamma',x_{n})\}-M(\log d+2D).
$$
This is (2).

(3) follows from (1) and (2) immediately.
\end{proof}

\vspace*{5pt}
\begin{proposition}~\label{main}
Suppose $\Gamma$ is an $f$-stable multi-curve satisfying $\Gamma=\Gamma_0$. Let
$x_{0}\in T_f$ and $x_{n}=\sigma_{f}^{n}(x)$, $n>0$, and
$D=d_{T}(x_{0},x_{1})$. Suppose $\text{min}_{\gamma\in \Gamma} w_{E}(\gamma,x_{0})\geq J_A$. Write
$\Gamma=\Gamma'\sqcup\Gamma^{''}$, where $\Gamma'=\Gamma_{Ob}$ is the union of
the irreducible component $\Gamma_j$ of $\Gamma$ for which
$\lambda(A_{\Gamma_j})\geq 1$. Then
\begin{itemize}
\item[1)] For all $\gamma \in \Gamma$,  $w_E(\gamma, x_n)\geq A$ for any $n\geq 0$.
\item[2)] For all $\gamma\in \Gamma'$, $\{w_E(\gamma, x_n)\}_{n\geq 0}$ is $(\kappa-2)$-quasi-nondecreasing.
 \item[3)] For all $\gamma\in \Gamma^{''}$ and all $n\geq 0$,
$$
w_E(\gamma, x_n)\geq \text{min}_{\gamma'\in \Gamma'}\{w_E(\gamma',x_n)\}-2-M(\log d+2D).
$$
\end{itemize}
\end{proposition}

\begin{proof}
From Lemma~\ref{comparision}, we have, for any $x\in T_{f}$,
$$
w_Q(\gamma,x)\leq w_E(\gamma,x)\leq w_Q(\gamma,x)+1$$
if $w_E(\gamma,x)\geq A$. If $\text{min}_\gamma w_{E}(\gamma,x_{0})\geq J_A$,
then $\text{min}_\gamma w_{Q}(\gamma,x_{0})\geq J_A-1$,
then by Proposition~\ref{main0}, for any $n\geq 0$ and $\gamma\in \Gamma$,
$w_Q(\gamma,x_n)\geq A$. Consequently, $w_E(\gamma,x_n)\geq A$. We get 1).

From 1), we have $|w_E(\gamma, x_n)-w_Q(\gamma,x_n)|<1$ for all $n\geq 0$. Then by Property~\ref{pr1} and Proposition~\ref{main0},
we have 2) and 3).
\end{proof}

\section{Proof of Theorem~\ref{ChenJiang}}
Choose any $x_0\in T_{f}$, we can find a $J\geq J_{A}$ such that $w_{E}(x_0)<C(J)$.
Without loss of generality, we assume that $J=J_{A}$.
Since $C(J)$ is an increasing function of $J$, we
have $w_{E}(x_{0})<C(J)$ for all $J\geq J_A$. Let $x_{n}=\sigma_{f}^{n} (x_{0})$, $n>0$, and $D=d_T(x_0,x_1)$.

Suppose that $f$ is not equivalent to a rational map. By Lemma~\ref{converges},
the sequence $\{w_E(x_n)\}_{n\geq 0}$ is unbounded. Thus there
exists $\gamma_k$ and $x_{n_k}$ with
$w_{E}(\gamma_k,x_{n_k})\rightarrow \infty$, as $k\to \infty$.

Fix $J>J_0=J_A+|A|$. Then $w_E(\gamma_k, x_{n_k})>E(J)=C(J)+2mD$ for some $k$.
So by Lemma~\ref{lbcurve}, the set of the finite
depth curves in $\Gamma_{J,x_{n_k}}$, denoted by $\Gamma_{J,x_{n_k},0}$, is nonempty.

Moreover,
if for some $n_0$, $\gamma\in \Gamma_{J,n_0,0}$, then $w_E(\gamma, x_{n_0})>a+J\geq
J_A$, which implies $w_E(\gamma, x_n)\geq A$ for all $n\geq n_0$ by
Proposition~\ref{main}. This implies that $\Gamma_J=\cup_n\Gamma_{J, x_n, 0}$
and $\mathcal{G}=\cup_{J\geq J_0}\Gamma_J$ are multi-curves, since
$\gamma\in \Gamma_J$ satisfies $w_E(\gamma,x_n)\geq A$ for all $n$
sufficiently large.

Since $w_E(\gamma_k, x_{n_k})\rightarrow \infty$, as $k\to \infty$, given any fixed
$J\geq J_0$, $w_E(\gamma_k, x_{n_k})\geq E(J)$ for infinitely many
$k$. Hence $\gamma_k\in \Gamma_J\subset \mathcal{G}$ infinitely
often. Since $\mathcal{G}$ is finite, for some $\gamma\in
\mathcal{G}$, we have $\gamma_k=\gamma$ for infinitely many $k$. Hence the
set $$\Gamma_u=\{\gamma \ | \ \{ w_E(\gamma, x_n)\}_{n\geq 0}\ \text{is
unbounded}\}$$ is nonempty.

\vspace*{5pt}
\begin{proposition}~\label{claim1}
$\Gamma_u=\cap_{J\geq J_0}\Gamma_J$.
\end{proposition}

\begin{proof}
The inclusion
$\cap_{J>J_0}\Gamma_J\subset \Gamma_u$ is clear. To see the other
inclusion, let $\gamma\in \Gamma_u$. Given $J$, there exists some $n$ such that
$w(\gamma, x_n)>E(J)$. By Lemma~\ref{lbcurve}, $\gamma\in
\Gamma_{J,x_n,0}$. Thus $\cap_{J>J_0}\Gamma_J\supset \Gamma_u$. This proves the proposition.
\end{proof}

\vspace*{5pt}
\begin{proposition}~\label{claim2}
$\Gamma_u=\Gamma_{J_c}$ for some $J_c\geq J_A$.
\end{proposition}

\begin{proof}
We prove it by contradiction. Since $\Gamma_u=\cap_{J\geq J_0}\Gamma_J$, for all $J\geq J_0$,
if $\Gamma_u\not=\Gamma_{J}$(I delete something here), then there exists a curve $\gamma_J$ such that
$\gamma_J\in \Gamma_J\subset\mathcal{G}$ but $\gamma_J\notin
\Gamma_u$.
Since $\mathcal{G}$ is finite, this implies that there is
some $\gamma\in \mathcal{G}$ such that $\gamma=\gamma_J\in \Gamma_J$ for
infinitely many $J$, while also $\gamma\notin \Gamma_u$. This is a
contradiction, since $\gamma\in \Gamma_J$ for infinitely many $J$ implies
that the sequence $\{w_E(\gamma,x_n)\}$ is unbounded. The contradiction proves the proposition.
\end{proof}

Now consider $\Gamma_u=\Gamma_{J_c}=\cup_n\Gamma_{J_c,x_n,0}$.

For each $k$ such that $\Gamma=\Gamma_{J_c,x_k,0}$ is nonempty,
applying Proposition~\ref{main}, we know that if $\gamma'\in \Gamma'$,
then the sequence $\{w_E(\gamma', x_n)\}_{n\geq 0}$ is both unbounded and
quasi-nondecreasing, so $w_{E}(\gamma', x_n)\rightarrow \infty$, as $n\to \infty$. 3) of
Proposition~\ref{main} implies that $w_{E}(\gamma, x_n)\rightarrow \infty$, as $n\to \infty$, for
all $\gamma\in \Gamma''$. Hence
$$
\Gamma_u=\{\gamma\ |\ w_E(\gamma,x_n)\rightarrow \infty \hbox{ as $n\to\infty$} \}.
$$

\vspace*{5pt}
\begin{proposition}~\label{claim3}
$\Gamma_u=\Gamma_{J_c, x_{n_c},0}$ for some $n=n_c$.
\end{proposition}

\begin{proof}
Since $\Gamma_u=\cup_n
\Gamma_{J_c,x_n,0}$, the inclusion
$\Gamma_{J_c,x_n,0}\subset\Gamma_u\subset \mathcal{G}$ holds for all $n$.

Since there are finitely many elements in $\mathcal{G}$,
there exists an $n_c$ such that for all $\gamma\in \Gamma_u$,
$$
w_{E}(\gamma,x_{n_c})>E(J_c).
$$
By Lemma~\ref{lbcurve}, we have
$\gamma\in \Gamma_{J_c,x_{n_c},0}$. Thus
$\Gamma_u=\Gamma_{J_c,n_c,0}$.
\end{proof}

From Proposition~\ref{claim3}, $\Gamma_u$ is a Thurston obstruction.
Furthermore, $\Gamma_u$ depends only on $f$ and is independent of the
initial point $x_{0}$, since for any $\gamma$, the map
$x\mapsto w_E(\gamma,x)$ is a Lipschitz map with Lipschitz constant $2$ (see Lemma~\ref{lipschitz}) and since
$\sigma_f$ decreases the Teichm\"uller distance $d_{T}$.

Finally, since
$$
w_Q(\gamma,x)\leq w_E(\gamma,x)\leq 1+w_Q(\gamma,x)
$$
if $w_E(\gamma,x)\geq A$ (refer to Remark 4.3), we have that
\begin{eqnarray*}
\Gamma_c &=& \{\gamma\ |\ w_Q(\gamma,x_n)\rightarrow \infty \hbox{ as $n\to\infty$}\}\cr
         &=& \{\gamma\ |\ w_E(\gamma,x_n)\rightarrow \infty \hbox{ as $n\to\infty$} \}
         = \Gamma_u.
\end{eqnarray*}
Therefore, $\Gamma_c$ is a Thurston obstruction. This completes the proof of Theorem~\ref{ChenJiang}.

\end{document}